\documentclass{article}
\usepackage{amsfonts}
\usepackage{latexsym}
\usepackage{mathrsfs}
\usepackage{amssymb}
\usepackage{amsmath}
\usepackage{amsthm}
\usepackage{indentfirst}
\usepackage{color}
\usepackage{float}
\usepackage{graphicx}

\allowdisplaybreaks
\newtheorem{theorem}{\color{black}\indent Theorem}[section]
\newtheorem{lemma}{\color{black}\indent Lemma}[section]

\newtheorem{remark}{\color{black}\indent Remark}[section]

\begin{document}
\title{\LARGE\bf Finite time blow-up for the heat flow of H-surface with constant mean curvature}
\author{Haixia Li}
 \date{}
 \maketitle

\footnotetext{\hspace{-1.9mm}$^\dag$Corresponding author.\\
Email addresses: lihaixia0611@126.com(H. Li).

\thanks{
$^*$Supported by NSFC (11626044), by Natural Science Foundation of Changchun Normal University(No. 2015-002) and
by Scientific Research Foundation for Talented Scholars of Changchun Normal University(No. RC2016-008).}}
\begin{center}
{\noindent\it\small School of Mathematics, Changchun Normal University,
 Changchun 130032, P.R. China}
\end{center}

\date{}
\maketitle

{\bf Abstract}\ In this paper, the authors consider an initial boundary value problem for the heat flow of
equation of surfaces with constant mean curvatures, which was investigated in
[On the heat flow of equation of surfaces of constant mean curvatures,
Manuscripta Mathematica, 2011, 134: 259-271] by Huang et al.,
where global well-posedness and finite time blow-up of regular solutions were obtained.
Their results are complemented in this paper in the sense that
some new conditions on the initial data are provided for the solutions to develop finite time singularity.

{\bf Keywords} H-surface; constant mean curvature; blow-up.

{\bf AMS Mathematics Subject Classification 2010:} 35K20, 58J35.

\section{Introduction}
\setcounter{equation}{0}

Suppose that $\Omega\subset \mathbb{R}^2$ is a bounded domain with smooth boundary $\partial \Omega$.
Given a continuous function $H: \mathbb{R}^3\rightarrow \mathbb{R}$, a map $u\in C^2(\Omega;\mathbb{R}^3)$ is called an
$H$-surface (parameterized over $\Omega$), if it satisfies
\begin{eqnarray}\label{1.1}
\Delta u=2H(u)u_x\wedge u_y,  \qquad (x,y)\in\Omega,
\end{eqnarray}
where $\wedge$ denotes the wedge product in $\mathbb{R}^3$. In fact,
if $u$ is a conformal representation of a surface $S$ in $\mathbb{R}^3$, i.e., $u_x\cdot u_y=0=|u_x|^2-|u_y|^2$,
then $H(u)$ is the mean curvature of $S$ at the point $u$.

After the pioneering work of Wente \cite{Wente} in 1969, the boundary value problem for the equation of $H$-surface \eqref{1.1}
has been studied extensively. We refer the interested readers to \cite{Brezis-Coron-1,Brezis-Coron-2,Hildebrandt,Struwe} for $H$-surface with
constant mean curvature and to \cite{Caldiroli-Musina,Duzaar-Grotowski,Rey} for $H$-surface with with variable $H$.

In this paper, we consider the parabolic counterpart of problem \eqref{1.1}, that is the following initial boundary value problem for the heat flow of
the equation of $H$-surface
\begin{eqnarray}\label{1.2}
\ \ \ \ \ \ \ \ \ \ \ \ \begin{cases}
u_t=\Delta u-2H(u)u_x\wedge u_y, & (x,y,t)\in\Omega\times(0,T),\\
u=\chi, & (x,y,t)\in\partial \Omega\times(0,T),\\
u(x,0)=u_0(x), &(x,y)\in \Omega,
\end{cases}
\end{eqnarray}
where $T\in(0,+\infty]$ is the maximal existence time of the solution $u(x,t)$, $u_0\in H^1(\Omega;\mathbb{R}^3)$,
$\chi\in H^{\frac{1}{2}}(\partial\Omega;\mathbb{R}^3)$ and $u_0\mid_{\partial\Omega}=\chi$.

In 1988, problem \eqref{1.2} (with the Dirichlet boundary condition replaced by free boundaries)
was employed by Struwe \cite{Struwe}, where the existence of surfaces of constant mean curvature $H=H_0$ was obtained,
under the assumption that $|H_0|\|u_0\|_{L^\infty(\Omega)}<1$.
Later, Rey \cite{Rey} showed that problem \eqref{1.2} with variable $H$ admits a unique global regular solution
$u\in C^{2+\alpha, 1+\frac{\alpha}{2}}(\overline{\Omega}\times(0,\infty);\mathbb{R}^3)$
when $u_0\in H^1(\Omega;\mathbb{R}^3)\cap L^\infty(\Omega;\mathbb{R}^3)$, $u_0\mid_{\partial\Omega}=\chi$ and
$\|H\|_{L^\infty(\mathbb{R}^3)}\|u_0\|_{L^\infty(\Omega)}<1$. When $u_0\in H^1(\Omega;\mathbb{R}^3)$ and
$\chi\in H^{\frac{3}{2}}(\partial\Omega;\mathbb{R}^3)$,  Chen and Levine \cite{Chen-Levine} removed the
assumption $\|H\|_{L^\infty(\mathbb{R}^3)}\|u_0\|_{L^\infty(\Omega)}<1$ and obtained the existence and uniqueness of short time
regular solution to problem \eqref{1.2}, with an additional assumption that
$$\int_\Omega |\nabla u|^2(\cdot,t)\leq \int_\Omega |\nabla u|^2(\cdot,s),\qquad 0\leq s\leq t.$$

There are also some works on global regular solutions to problem \eqref{1.2} with general $H$.
For example, for an arbitrary Lipschitz function $H$,
Wang \cite{Wang} proved that if $u\in H^1(\Omega\times(0,\infty);\mathbb{R}^3)$ is a weak solution to \eqref{1.2},
then $u\in C^{2,\alpha}(\Omega\times(0,\infty)\setminus\Sigma;\mathbb{R}^3)$, where
$$\Sigma=\bigcup_{t>0}\Sigma_t\subset \Omega\times(0,\infty)$$
is a closed subset, whose Lebesgue measure is zero and $\Sigma_t\subset \Omega\times\{t\}$ is finite for almost all $t>0$.
As for the finite time singularity of regular solutions to problem \eqref{1.2},
Huang et al. \cite{Huang} showed, when $H$ is a nontrivial constant and $\chi=0$,
that the regular solution to problem \eqref{1.2} blows up in finite time
when the initial energy is ``subcritical".

The aim of this paper is to reconsider the finite time singularity of regular solutions to
problem \eqref{1.2} and to give some new conditions on the initial data to assure finite time blow-up.
We have to admit that we do not invent any new strategies or new
methods to prove the finite time blow up results, but only apply the concavity argument (\cite{Han,Levine}) to a newly constructed auxiliary functional.
Henceforth we always assume that
\begin{equation}\label{1.3}
H\equiv H_0\in \mathbb{R}\setminus\{0\},\qquad \chi=0.
\end{equation}

The main results of this paper will be stated and proved in the next section.

\par
\section{Main results}
\setcounter{equation}{0}

In order to describe our main results, we first introduce some notations and lemmas.
If no confusion arises, we will always write
$\int_\Omega\cdot\ {\rm d}x{\rm d}y$ and $\int_0^t\int_\Omega\cdot\ {\rm d}x{\rm d}y{\rm d}s$
as $\int_\Omega\cdot$ and $\int_0^t\int_\Omega\cdot$, respectively.
We denote by $\|\cdot\|_{r}$ the $L^r(\Omega;\mathbb{R}^3)$ norm for $1\leq r\leq \infty$
and by $\|\cdot\|$ the Dirichlet norm in $H_0^1(\Omega;\mathbb{R}^3)$,
i.e., $\|u\|=\|\nabla u\|_2$ for any $u\in H_0^1(\Omega;\mathbb{R}^3)$.
The potential energy functional and the Nehari's functional associated with problem \eqref{1.2}
are defined, respectively, by
\begin{equation}\label{e}
E(u)=\frac{1}{2}\|u\|^2+\frac{2H_0}{3}\int_\Omega u\cdot u_x\wedge u_y,
\end{equation}
\begin{equation}\label{n}
N(u)=\|u\|^2+2H_0\int_\Omega u\cdot u_x\wedge u_y.
\end{equation}

For any regular solution $u(x,t)$ to problem \eqref{1.2},
the functional $E(u)$ satisfies the following energy identity.

\begin{lemma}\label{lem-ei}
(\cite{Huang}) For $0<T\leq \infty$, suppose that $u:\Omega\times[0,T)\rightarrow \mathbb{R}^3$ is a regular solution to
problem \eqref{1.2}. Then it holds
\begin{equation}\label{ei}
\int_{t_1}^{t_2}\int_\Omega|u_t|^2+E(u(t_2))=E(u(t_1)),\quad \forall\ 0\leq t_1\leq t_2<T.
\end{equation}
\end{lemma}

The main result of this paper is proved with the help of the following concavity lemma \cite{Levine},
whose proof can be found in \cite{Li}

\begin{lemma}\label{concave}
Suppose that a positive, twice-differentiable function $\psi(t)$ satisfies the inequality
$$\psi''(t)\psi(t)-(1+\theta)(\psi'(t))^2\geq0,$$
where $\theta>0$. If $\psi(0)>0$, $\psi'(0)>0$, then $\psi(t)\rightarrow\infty$ as $t\rightarrow t_*\leq t^*=\frac{\psi(0)}{\theta\psi'(0)}$.
\end{lemma}

We are now in the position to state and prove the main results of this paper.
\begin{theorem}\label{main}
Let assumption \eqref{1.3} hold. Assume that $u_0\in H_0^1(\Omega;\mathbb{R}^3)$ satisfying
\begin{equation}\label{initial}
E(u_0)<\frac{\lambda_1}{6}\|u_0\|_2^2,
\end{equation}
were $\lambda_1>0$ is the first eigenvalue of $-\Delta$ in $\Omega$ under homogeneous Dirichlet boundary condition.
Then the local regular solution $u(x,t)$ to problem \eqref{1.2} blows up in finite time.
Moreover, the upper bound for $T$ has the following form:
$T\leq \dfrac{16\|u_0\|_2^2}{\lambda_1\|u_0\|_2^2-6E(u_0)}$.
\end{theorem}

\begin{proof}
The proof will be divided into two steps.

\noindent{\bf Step I: Finite time blow-up for local regular solution.}

Suppose by contradiction that the local regular solution $u(x,t)$ exists globally.
Then $\|u(\cdot,t)\|_2$ is well-defined for all $t\geq0$.
Without loss of generality, we may assume that $E(u(t))>0$ for all $t\geq0$.
Otherwise by Theorem 1.1 in \cite{Huang} we know that $u(x,t)$ blows up in finite time.

First, multiplying both sides of the first equation in \eqref{1.2} by $u$ and integrating over $\Omega$ one obtains
\begin{equation}\label{monotonicity}
\dfrac{d}{dt}\|u(t)\|_2^2=-2\|u\|^2-4H_0\int_\Omega u\cdot u_x\wedge u_y=-2N(u(t)),\qquad t\in(0,\infty).
\end{equation}
From \eqref{e}, \eqref{n}  and the definition of $\lambda_1$ we have
\begin{equation}\label{n-e}
N(u)=3\Big[E(u)-\frac{1}{6}\|u\|^2\Big]\leq 3\Big[E(u)-\frac{\lambda_1}{6}\|u\|_2^2\Big].
\end{equation}
Substituting \eqref{n-e} into \eqref{monotonicity} to yield
\begin{equation}\label{monotonicity-2}
\dfrac{d}{dt}\|u(t)\|_2^2\geq\lambda_1\|u(t)\|_2^2-6E(u(t)),\qquad t\in(0,\infty).
\end{equation}
Combining \eqref{monotonicity-2} with \eqref{ei} one sees that
\begin{equation}\label{monotonicity-3}
\dfrac{d}{dt}\Big[\|u(t)\|_2^2-\dfrac{6}{\lambda_1}E(u(t))\Big]
=\dfrac{d}{dt}\|u(t)\|_2^2+\dfrac{6}{\lambda_1}\|u_t(t)\|_2^2\geq\lambda_1\Big[\|u(t)\|_2^2-\dfrac{6}{\lambda_1}E(u(t))\Big].
\end{equation}
Recalling that $\|u_0\|_2^2-\dfrac{6}{\lambda_1}E(u_0)>0$, we obtain,
by applying Gronwall's inequality to \eqref{monotonicity-3}, that
\begin{equation}\label{monotonicity-4}
\|u(t)\|_2^2-\dfrac{6}{\lambda_1}E(u(t))\geq \Big[\|u_0\|_2^2-\dfrac{6}{\lambda_1}E(u_0)\Big]e^{\lambda_1t}.
\end{equation}
Therefore,
\begin{equation}\label{monotonicity-5}
\begin{split}
\|u(t)\|_2^2&\geq \Big[\|u_0\|_2^2-\dfrac{6}{\lambda_1}E(u_0)\Big]e^{\lambda_1t}+\dfrac{6}{\lambda_1}E(u(t))\\
&\geq\Big[\|u_0\|_2^2-\dfrac{6}{\lambda_1}E(u_0)\Big]e^{\lambda_1t},\qquad t\in(0,\infty),
\end{split}
\end{equation}
since $E(u(t))\geq0$ on $[0,\infty)$.

On the other hand, in view of Minkowski inequality, H\"{o}lder inequality, the energy identity \eqref{ei}
and the fact that $E(u(t))>0$ for $t\geq0$ we obtain
\begin{equation*}
\begin{split}
\|u(t)\|_2\leq&\|u_0\|_2+\|u(t)-u_0\|_2=\|u_0\|_2+\|\int_0^t u_\tau{\rm d}\tau\|_2\\
\leq&\|u_0\|_2+\int_0^t \|u_\tau\|_2{\rm d}\tau\leq\|u_0\|_2
+\Big(\int_0^t \|u_\tau\|_2^2{\rm d}\tau\Big)^{1/2}t^{1/2}\\
=&\|u_0\|_2+[E(u_0)-E(u(t))]^{1/2}t^{1/2}\\
\leq&\|u_0\|_2+\sqrt{E(u_0)t},
\end{split}
\end{equation*}
which contradicts with \eqref{monotonicity-5} for $t$ sufficiently large.
Therefore, the local regular solution $u(x,t)$ must blow up in finite time.

\noindent{\bf Step II: Upper bound for the blow-up time.}

Form now on, we denote by $T>0$ the blow-up time of $u(x,t)$, which is finite by Step I.
From \eqref{n-e} and \eqref{monotonicity-4} we see that $N(u(t))<0$ on $[0,T)$,
which, together with \eqref{monotonicity}, implies that $\|u(t)\|_2^2$ is strictly increasing on $[0,T)$.

For any $\beta>0$ and $\sigma>0$, we define the auxiliary functional
\begin{equation}\label{F}
F(t)=\int_0^t\|u(s)\|_2^2\mathrm{d}s+(T-t)\|u_0\|_2^2+\beta(t+\sigma)^2,\qquad t\in[0,T).
\end{equation}
Then $F(0)=T\|u_0\|_2^2+\beta\sigma^2>0$. Taking derivative with respect to $t$ we obtain
\begin{equation}\label{F1}
\begin{split}
F'(t)&=\|u(t)\|_2^2-\|u_0\|_2^2+2\beta(t+\sigma)=\int_0^t\frac{d}{ds}\|u(s)\|_2^2+2\beta(t+\sigma)\\
&=2\int_0^t\int_\Omega u\cdot u_s+2\beta(t+\sigma),
\end{split}
\end{equation}
and $F'(0)=2\beta\sigma>0$. Taking derivative again and recalling \eqref{e}, \eqref{n} and \eqref{ei},
one arrives at
\begin{equation}\label{F2}
\begin{split}
F''(t)=&2\int_\Omega u\cdot u_t+2\beta=-2N(u(t))+2\beta(>0)\\
\geq& -6E(u(t))+\|u(t)\|^2+2\beta\\
   =& -6E(u_0)+6\int_0^t\|u_s(s)\|_2^2\mathrm{d}s+\|u(t)\|^2+2\beta.
\end{split}
\end{equation}
Noticing that $F''(t)>0$ on $[0,T)$, $F'(t)$ is monotone increasing on $[0,T)$ and $F'(t)\geq F'(0)>0$.
Therefore, $F(t)$ is strictly increasing on $[0,T)$.

In view of \eqref{F}-\eqref{F2} and recalling the monotonicity of $\|u(t)\|_2^2$, we have,
for any $t\in[0,T)$ and $\beta\in\Big(0,\frac{\lambda_1\|u_0\|_2^2-6E(u_0)}{4}\Big]$, that
\begin{equation}\label{concave-1}
\begin{split}
&F(t)F''(t)-\dfrac{3}{2}(F'(t))^2=F(t)F''(t)-6\Big[\int_0^t\int_\Omega u\cdot u_s+\beta(t+\sigma)\Big]^2\\
=&F(t)F''(t)+6\Big[\eta(t)-(F-(T^*-t)\|u_0\|_2^2)(\int_0^t\|u_s(s)\|_2^2\mathrm{d}s+\beta)\Big]\\
\geq&F(t)F''(t)-6F(t)\Big[\int_0^t\|u_s(s)\|_2^2\mathrm{d}s+\beta\Big]\\
\geq&F(t)\Big[\|u(t)\|^2-6E(u_0)-4\beta\Big]\\
\geq&F(t)\Big[\lambda_1\|u(t)\|_2^2-6E(u_0)-4\beta\Big]\\
\geq&F(t)\Big[\lambda_1\|u_0\|_2^2-6E(u_0)-4\beta\Big]\\
\geq&0,
\end{split}
\end{equation}
where
$$\eta(t)=\Big[\int_0^t\|u(s)\|_2^2\mathrm{d}s+\beta(t+\sigma)^2\Big]\Big[\int_0^t\|u_s(s)\|^2\mathrm{d}s+\beta\Big]
-\Big[\int_0^t\int_\Omega u\cdot u_s+\beta(t+\sigma)\Big]^2,\quad t\in[0,T),$$
the nonnegativity of which can be verified by simply using Cauchy-Schwarz inequality
and H\"{o}lder's inequality.
Now applying Lemma \ref{concave} to $F(t)$ one sees that
\begin{equation}\label{F3}
\lim\limits_{t\rightarrow T}F(t)=\infty,
\end{equation}
where
\begin{equation*}
T\leq\dfrac{2F(0)}{F'(0)}=\dfrac{2(T\|u_0\|_2^2+\beta\sigma^2)}{2\beta\sigma}=\dfrac{\|u_0\|_2^2}{\beta\sigma}T+\sigma,
\end{equation*}
or
\begin{equation}\label{T1}
T\Big(1-\dfrac{\|u_0\|_2^2}{\beta\sigma}\Big)\leq \sigma,
\end{equation}
for any $\beta\in\Big(0,\frac{\lambda_1\|u_0\|_2^2-6E(u_0)}{4}\Big]$ and $\sigma>0$.

Fix a $\beta_0\in\Big(0,\frac{\lambda_1\|u_0\|_2^2-6E(u_0)}{4}\Big]$.
Then we have $0<\dfrac{\|u_0\|_2^2}{\beta_0\sigma}<1$ for any $\sigma\in(\dfrac{\|u_0\|_2^2}{\beta_0},+\infty)$,
which, combined with \eqref{T1}, guarantees that
\begin{equation}\label{T2}
T\leq \dfrac{\beta_0\sigma^2}{\beta_0\sigma-\|u_0\|_2^2}.
\end{equation}
Minimizing the right hand side term in \eqref{T2} for $\sigma\in\Big(\dfrac{\|u_0\|_2^2}{\beta_0},+\infty\Big)$ yields
\begin{equation}\label{T3}
T\leq\dfrac{4\|u_0\|_2^2}{\beta_0},\quad \beta_0\in\Big(0,\frac{\lambda_1\|u_0\|_2^2-6E(u_0)}{4}\Big].
\end{equation}
Minimizing the right hand side term in \eqref{T3} for $\beta_0\in\Big(0,\frac{\lambda_1\|u_0\|_2^2-6E(u_0)}{4}\Big]$ we finally obtain
\begin{equation*}
T\leq \dfrac{16\|u_0\|_2^2}{\lambda_1\|u_0\|_2^2-6E(u_0)}.
\end{equation*}
The proof is complete.
\end{proof}

\begin{remark}
In \cite{Huang}, the authors showed that the local regular solution
$u(x,t)$ to problem \eqref{1.2} blows up in finite time when $E(u_0)<\dfrac{4\pi}{3H_0^2}$
and $|\int_\Omega u_0\cdot u_{0x}\wedge u_{0y}|>\dfrac{4\pi}{|H_0|^3}$
(the latter is unnecessary when $E(u_0)<0$). It is easy to verify that
this condition and assumption \eqref{initial} can not be deduced from
each other. Therefore, we obtain a new blow-up criterion for problem \eqref{1.2}.
\end{remark}

{\bf Acknowledgement}\\
The author would like to express her sincere gratitude to Professor Wenjie Gao in Jilin University for his enthusiastic
guidance and constant encouragement.

\end{document}